\documentclass[leqno,11pt,a4paper]{amsart}
\usepackage{graphicx}
\usepackage[usenames,dvipsnames]{color}
\usepackage{hyperref}
\usepackage{mathabx}
\usepackage{mleftright}
\newtheorem{thm}{Theorem}[section]
\newtheorem{lemma}[thm]{Lemma}

\newtheorem{prop}[thm]{Proposition}
\theoremstyle{definition}
\newtheorem{example}[thm]{Example}
\newtheorem{remark}[thm]{Remark}

\newtheorem{definition}[thm]{Definition}

\numberwithin{equation}{section}
\long\def\blankfootnotetext#1{\begingroup\def\thefootnote{\fnsymbol{footnote}}\footnotetext{#1}\endgroup}
\newcommand{\orig}{\mathbf{0}}
\newcommand{\Z}{\mathbb{Z}}
\newcommand{\Q}{\mathbb{Q}}
\newcommand{\R}{\mathbb{R}}
\newcommand{\Proj}{\mathbb{P}}
\newcommand{\Hom}[1]{\mathrm{Hom}\mleft({#1}\mright)}

\newcommand{\abs}[1]{\left\vert{#1}\right\vert}
\newcommand{\mult}[1]{\mathrm{mult}\mleft({#1}\mright)}
\newcommand{\intr}[1]{\mathrm{int}\mleft({#1}\mright)}
\newcommand{\V}[1]{\mathrm{vert}\mleft({#1}\mright)}

\renewcommand{\gcd}[1]{\mathrm{gcd}\mleft\{{#1}\mright\}}
\renewcommand{\min}[1]{\mathrm{min}\mleft\{{#1}\mright\}}
\renewcommand{\max}[1]{\mathrm{max}\mleft\{{#1}\mright\}}
\newcommand{\dual}[1]{{#1}^\vee}
\newcommand{\bdual}[1]{\dual{\mleft({#1}\mright)}}
\newcommand{\conv}[1]{\mathrm{conv}\mleft({#1}\mright)}
\newcommand{\sconv}[1]{\mathrm{conv}\mleft\{{#1}\mright\}}
\newcommand{\scone}[1]{\mathrm{cone}\mleft\{{#1}\mright\}}
\newcommand{\linspan}[1]{\left<{#1}\right>}
\newcommand{\lineseg}[2]{\overline{{#1}{#2}}}
\newcommand{\hmin}{{h_\mathrm{min}}}
\newcommand{\hmax}{{h_\mathrm{max}}}
\newcommand{\umin}{{u_\mathrm{min}}}

\newcommand{\mut}{\mathrm{mut}}
\newcommand{\Mm}{M^-}
\newcommand{\Mp}{M^+}
\newcommand{\NQ}{N_\Q}
\newcommand{\MQ}{M_\Q}
\newcommand{\modb}[1]{\left(\mathrm{mod}\ {#1}\right)}
\graphicspath{{images/}}
\begin{document}
\author[M.~E.~Akhtar]{Mohammad E.~Akhtar}
\author[A.~M.~Kasprzyk]{Alexander M.~Kasprzyk}
\address{Department of Mathematics\\Imperial College London\\London, SW$7$\ $2$AZ\\UK}
\email{mohammad.akhtar03@imperial.ac.uk}
\email{a.m.kasprzyk@imperial.ac.uk}
\blankfootnotetext{2010 \emph{Mathematics Subject Classification}: 52B20 (Primary); 14J45, 11D99 (Secondary).}
\title{Mutations of fake weighted projective planes}
\begin{abstract}
In previous work by Coates, Galkin, and the authors, the notion of mutation between lattice polytopes was introduced. Such a mutation gives rise to a deformation between the corresponding toric varieties. In this paper we study one-step mutations that correspond to deformations between weighted projective planes, giving a complete characterisation of such mutations in terms of $T$-singularities. We show also that the weights involved satisfy Diophantine equations, generalising results of Hacking--Prokhorov.
\end{abstract}
\maketitle
\section{Introduction}
In~\cite{ACGK12} we described a combinatorial notion of mutation between convex lattice polytopes. In this paper we begin to explore the geometry behind this idea. Given a convex lattice polytope $P$ containing the origin and with primitive vertices, there is a corresponding toric variety $X$ defined by the spanning fan of $P$. A mutation between polytopes $P$ and $Q$ determines a deformation between $X_P$ and $X_Q$~\cite{Ilt12}. Our main result characterises mutations between triangles; thus we characterise certain deformations, over $\Proj^1$, with fibers given by fake weighted projective planes. We recover and generalise certain results of Hacking and Prokhorov~\cite[Theorem~4.1]{HP10} connecting the fake weighted projective planes with $T$-singularities to solutions of Markov-type equations. We prove the following:

\begin{prop}\label{prop:only_mutate_T_sings}
Let $X=\Proj(\lambda_0,\lambda_1,\lambda_2)$ be a weighted projective plane. Up to reordering of the weights, there exists a one-step mutation to a weighted projective plane $Y$ if and only if $\frac{1}{\lambda_0}(\lambda_1,\lambda_2)$ is a $T$-singularity. When this is the case, $Y=\Proj\left(\lambda_1,\lambda_2,\frac{(\lambda_1+\lambda_2)^2}{\lambda_0}\right)$. More generally, there exists a one-step mutation from the fake weighted projective plane $X/(\Z/n)$ to the fake weighted projective plane $Y/(\Z/n')$ only if $n=n'$ and $\frac{1}{\lambda_0}(\lambda_1,\lambda_2)$ is a $T$-singularity.
\end{prop}

In Proposition~\ref{prop:mult_preserved} we associate to a weighted projective plane $X$ a Diophantine equation
\begin{equation}\label{eq:intro_diophantine}
mx_0x_1x_2 = k(c_0x_0^2 + c_1x_1^2 + c_2x_2^2).
\end{equation}
The weights $(\lambda_0,\lambda_1,\lambda_2)$ of $X$ correspond to a
solution $(a_0,a_1,a_2)$, where $\lambda_i=c_ia_i^2$, $i=0,1,2$, and
the degree of $X$ is given by
$$(-K_X)^2=\frac{m^2}{c_0c_1c_2k^2}.$$
One-step mutations of $X$ correspond to transformations of the solutions to~\eqref{eq:intro_diophantine}, and all such solutions can be generated from the so-called minimal weights by mutation.

When $X=\Proj^2$, equation~\eqref{eq:intro_diophantine} becomes the celebrated Markov equation~\cite{Mar88}. Certain other special cases were studied by Rosenberger~\cite{Ros79}. These cases all have finitely many minimal weights. In~\S\ref{sec:infinite_min_weights} we give an example where the corresponding Diophantine equation has infinitely many minimal weights.

\section{Mutations of Fano polytopes}\label{sec:mutations_review}
Let $N\cong\Z^n$ be a lattice with dual $M:=\Hom{N,\Z}$. A lattice polytope $P\subset\NQ:=N\otimes_\Z\Q$ is called \emph{Fano} if it satisfies three conditions:
\begin{enumerate}
\item
$P$ is of maximum dimension, $\dim{P}=\dim{N}$;
\item\label{defn:fano_orign}
The origin is contained in the strict interior of $P$, $\orig\in\intr{P}$;
\item
The vertices $\V{P}$ of $P$ are primitive lattice points, i.e.~for any $v\in\V{P}$ there are no other lattice points on the line segment $\lineseg{\orig}{v}$ joining $v$ and the origin.
\end{enumerate}
The dual of $P$ is defined to be the polyhedron
$$\dual{P}:=\{u\in\MQ\mid u(v)\geq -1\text{ for all }v\in P\}\subset\MQ.$$
By condition~\eqref{defn:fano_orign} this is a polytope with $\orig\in\intr{\dual{P}}$, although it need not be a lattice polytope. See~\cite{KN12} for an overview of Fano polytopes.

We briefly recall the notation of~\cite[\S3]{ACGK12}. Any choice of primitive vector $w\in M$ determines a lattice height function $w:N\rightarrow\Z$ which naturally extends to $\NQ\rightarrow\Q$. A subset $S\subset\NQ$ is said to lie at height $h\in\Q$ with respect to $w$ if $w(S):=\{w(s)\mid s\in S\}=\{h\}$; we write $w(S)=h$. The set of all points of $\NQ$ lying at height $h$ with respect to a given $w$ is an affine hyperplane $H_{w,h}:=\{v\in\NQ\mid w(v)=h\}$. In particular,
$$w_h(P):=\conv{H_{w,h}\cap P\cap N}\subset\NQ$$
will denote the (possibly empty) convex hull of all lattice points in $P$ at height $h$.

Define
$$\hmin:=\min{w(v)\mid v\in P},\qquad\hmax:=\max{w(v)\mid v\in P}.$$
Since $P$ is a lattice polytope, both $\hmin$ and $\hmax$ are integers. Condition~\eqref{defn:fano_orign} guarantees that $\hmin<0$ and $\hmax>0$.

\begin{definition}\label{defn:factor}
A \emph{factor} of $P$ with respect to $w$ is a lattice polytope $F\subset\NQ$ satisfying:
\begin{enumerate}
\item
$w(F) = 0$;
\item
For every integer $h$, $\hmin\leq h<0$, there exists a (possibly empty) lattice polytope $G_h\subset\NQ$ at height $h$ such that
$$H_{w,h}\cap\V{P}\subseteq G_h+(-h)F\subseteq w_h(P).$$
\end{enumerate}
\end{definition}

Note that, for given polytope $P\subset\NQ$ and width vector $w\in M$, a factor $F$ need not exist. When a factor does exist we make the following construction:

\begin{definition}[\protect{\cite[Definition~5]{ACGK12}}]\label{defn:mutation}
Let $P\subset\NQ$ be a polytope with width vector $w\in M$, factor $F$, and polytopes $\{G_h\}$. We define the corresponding \emph{combinatorial mutation} to be the convex lattice polytope
$$\mut_w(P,F;\{G_h\}):=\conv{\bigcup_{h=\hmin}^{-1}G_h\cup\bigcup_{h = 0}^\hmax(w_h(P) + hF)}\subset\NQ.$$
For brevity we will refer to a combinatorial mutation simply as a \emph{mutation}.
\end{definition}

\noindent
We summarise the key properties of mutation~\cite{ACGK12}:
\begin{enumerate}
\item
Since for any $v\in N$ such that $w(v) = 0$ we have that
$$\mut_w(P,F;\{G_h\})\cong\mut_w(P, v+F;\{G_h + hv\}),$$
we need only consider factors $F$ up to translation. In particular, choosing $F$ to be a point leaves $P$ unchanged (up to isomorphism).
\item
If $\{G_h\}$ and $\{G_h^{\prime}\}$ are any two collections of polytopes for a factor $F$, then
$$\mut_w(P,F;\{G_h\})\cong\mut_w(P,F;\{G_h^{\prime}\}).$$
Thus the choice of collection $\{G_h\}$ is irrelevant and we write $\mut_w(P,F)$.
\item
$P$ is a Fano polytope if and only if $\mut_w(P,F)$ is a Fano polytope.
\item
Let $Q:=\mut_w(P,F)$. Then $\mut_{-w}(Q,F)=P$, so mutations are invertible.
\end{enumerate}
In~\cite{ACGK12} it was also shown that mutations have a natural description as a piecewise linear transformation of the lattice $M$. We require the following definition.

\begin{definition}
The \emph{inner normal fan} in $M$ of a polytope $F\subset\NQ$ is generated by the cones $\sigma_{v_F}$ consisting of those linear functions which are minimal on a given vertex $v_F$ of $F$. That is,
$$\sigma_{v_F}:=\left\{u\in\MQ\mid u(v_F)=\min{u(v')\mid v'\in F}\right\}.$$
\end{definition}

\begin{enumerate}
\setcounter{enumi}{4}
\item\label{item:mutation_in_M}
A mutation of $P\subset\NQ$ induces a piecewise linear transformation $\varphi$ of $\MQ$ such that $\bdual{\varphi(\dual{P})}=\mut_w(P,F)$, given by
$$\varphi:u\mapsto u -\umin w,\qquad u\in\MQ,$$
where $\umin:=\min{u(v_F)\mid v_F\in\V{F}}$. The inner normal fan of $F\subset\NQ$ determines a chamber decomposition of $\MQ$, and $\varphi$ acts as a linear transformation on the interior of each maximal dimensional cone of this fan.
\item\label{item:degree_preserved}
As a consequence of~\eqref{item:mutation_in_M}, the toric varieties $X_P$ and $X_Q$ defined by the spanning fans of $P$ and $Q:=\mut_w(P,F)$ have the same degree (in fact they have the same Hilbert series).
\end{enumerate}

\begin{example}\label{eg:P2_to_P114}
Consider the triangle $P=\sconv{(1,-1),(-1,2),(0,-1)}\subset\NQ$ corresponding to the toric variety $\Proj^2$. Let $w=(0,1)\in M$ and set $F=\sconv{\orig,(1,0)}\subset\NQ$. This defines a mutation from $P$ to the triangle $Q=\sconv{(1,2),(-1,2),(0,-1)}\subset\NQ$, as illustrated in Figure~\ref{fig:P2_to_P114}. On the dual side, this corresponds to a piecewise linear map $\varphi:u\mapsto uM_\sigma$ for $u=(\alpha,\beta)\in\MQ$, where
$$
M_\sigma=\left\{
\begin{array}{ll}
\small\begin{pmatrix}
1&0\\0&1
\end{pmatrix}&\text{ if }\alpha\geq 0,\\
\small\begin{pmatrix}
1&-1\\0&1
\end{pmatrix}&\text{ otherwise.}
\end{array}
\right.
$$
In particular, $\varphi(\dual{P})=\dual{Q}$.
\end{example}

\begin{figure}[tb]
$$
\begin{array}{rccc}
\NQ:&
\begin{minipage}[c]{38pt}\includegraphics{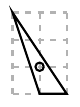}\end{minipage}&
\longmapsto&
\begin{minipage}[c]{38pt}\includegraphics{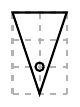}\end{minipage}\\
\MQ:&
\begin{minipage}[c]{90pt}\includegraphics{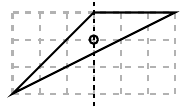}\end{minipage}&
\longmapsto&
\begin{minipage}[c]{90pt}\includegraphics{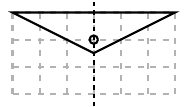}\end{minipage}
\end{array}
$$
\caption{A mutation from the triangle associated with $\Proj^2$ to the triangle associated with $\Proj(1,1,4)$.}\label{fig:P2_to_P114}
\end{figure}

Mutations are particularly simple in the two-dimensional case. In this setting, $w\in M$ defines a non-trivial mutation of $P\subset\NQ$ if and only if $w\in\{\overline{u}\mid u\in\V{\dual{P}}\}\subset M$, where $\overline{u}\in M$ is the unique primitive lattice vector on the ray passing through $u$. Nontrivial factors $F\subset\NQ$ are just line segments, so it suffices to restrict attention to those $F$ which have vertex set $\{\orig,f\}$, for some $f\in N$ with $w(f) = 0$. The inner normal fan of any factor $F$ of $P$ with respect to a given $w$ is just the linear subspace of $\MQ$ spanned by $w$. This divides $\MQ$ into two chambers; the piecewise linear transformation $\varphi$ acts trivially in one of the chambers, and as $u\mapsto u-u(f)w$ in the other.

\section{One-step mutations of triangles}\label{sec:mutation_results}
Set $N\cong\Z^2$ and let $P:=\sconv{v_0,v_1,v_2}\subset\NQ$ be a Fano triangle. Since $\orig\in\intr{P}$ there exists a (unique) choice of coprime positive integers $\lambda_0,\lambda_1,\lambda_2\in\Z_{>0}$ with $\lambda_0v_0+\lambda_1v_1+\lambda_2v_2=\orig$. The  projective toric surface $X$ given by the spanning fan of $P$ has Picard rank $1$, and is called a \emph{fake weighted projective plane} with weights $(\lambda_0,\lambda_1,\lambda_2)$; $X$ is the quotient of $\Proj(\lambda_0,\lambda_1,\lambda_2)$ by the action of a finite group of order $\mult{X}$ acting freely in codimension one~\cite{Con02,Buc08,Kas08b}.

\begin{remark}
Since the vertices of $P$ are primitive, the weights $(\lambda_0,\lambda_1,\lambda_2)$ are \emph{well-formed}: that is, $\gcd{\lambda_i,\lambda_j}=1$, $i\neq j$. In this paper we will always require that weights are well-formed.
\end{remark}

\begin{definition}\label{defn:onestep_mutation}
We say that a fake weighted projective plane $Y$ with defining Fano triangle $Q\subset\NQ$ is obtained from $X$ by a \emph{one-step mutation} if $Q\cong\mut_w(P,F)$ for some choice of $w$ and factor $F$.
\end{definition}

\subsection{One-step mutations in $\MQ$ and weights}\label{subsec:weights}
\begin{figure}[tb]
\centering
\includegraphics[scale=0.9]{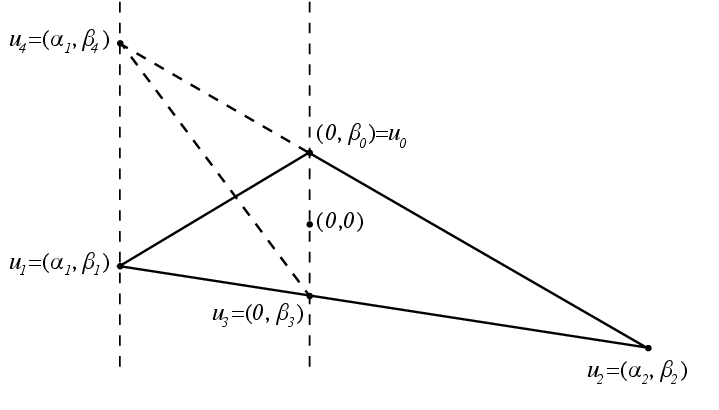}
\caption{A one-step mutation, depicted in $\MQ$, of the triangle $\sconv{u_0,u_1,u_2}$ to the triangle $\sconv{u_2,u_3,u_4}$.}\label{fig:mutation}
\end{figure}

First we address how the weights $(\lambda_0,\lambda_1,\lambda_2)$ associated with a Fano triangle $T\subset\NQ$ transform under mutation. We will require the following fact (see, for example,~\cite[Lemma~5.3]{Con02}): Let $\dual{T}=\sconv{u_0,u_1,u_2}$ by the triangle in $\MQ$ dual to $T$. Then, after possible reordering, $\lambda_0u_0+\lambda_1u_1+\lambda_2u_2=\orig$. Hence the weights of $T$ and the weights of $\dual{T}$ are equivalent.

\begin{prop}\label{prop:weights_transform}
Let $X$ be a fake weighted projective plane with weights $(\lambda_0,\lambda_1,\lambda_2)$. Suppose there exists a one-step mutation to a fake weighted projective plane $Y$. Then, up to relabelling, $\lambda_0\divides(\lambda_1+\lambda_2)^2$ and $Y$ has weights
$$\left(\lambda_1,\lambda_2,\frac{(\lambda_1+\lambda_2)^2}{\lambda_0}\right).$$
\end{prop}
\begin{proof}
Consider a lattice triangle $T_1\subset\NQ$, $\orig\in\intr{T_1}$, and suppose that there exists a width vector $w\in M$ and factor $F\subset\NQ$, $w(F)=0$, such that the mutation $T_2=\mut_w(T_1,F)$ is also a triangle. Without loss of generality we can assume that $w=(0,1)\in M$ and $F=\sconv{\orig,(a,0)}$ for some $a\in\Z_{>0}$. The mutation corresponds to a piecewise linear action on $\MQ$ via $u\mapsto uM_\sigma$ given by
$$
M_\sigma=\left\{\begin{array}{ll}
\small\begin{pmatrix}
1&0\\0&1
\end{pmatrix}&\text{ if }u\in\Mp,\\
\small\begin{pmatrix}
1&-a\\0&1
\end{pmatrix}&\text{ otherwise,}
\end{array}\right.
$$
where $\Mp$ is the half-space $\{(\alpha,\beta)\in\MQ\mid\alpha>0\}$. Let $\dual{T_1}=\sconv{u_0,u_1,u_2}\subset\MQ$ be the (possibly rational) triangle dual to $T_1$, where $u_2\in\Mp$ and so is fixed under the action of the mutation, and $u_1\in\Mm:=\{(\alpha,\beta)\in\MQ\mid\alpha<0\}$. Since $\dual{T_2}\subset\MQ$ is also a triangle, the only possibility is that $u_0$ lies on the line $\linspan{w}:=\{\gamma w\in\MQ\mid\gamma\in\Q\}$, $\dual{T_2}=\sconv{u_2,u_3,u_4}$ where $u_0$ is contained in the line segment $\lineseg{u_2}{u_4}$ joining $u_2$ and $u_4$, and $u_3$ is contained in the line segment $\lineseg{u_1}{u_2}$. This situation is illustrated in Figure~\ref{fig:mutation}.

Since $\orig\in\dual{T_1}$ there exist unique weights $(\lambda_0,\lambda_1,\lambda_2)\in\Z_{>0}^3$, $\gcd{\lambda_0,\lambda_1,\lambda_2}=1$, such that
\begin{equation}\label{eq:barycentre}
\lambda_0u_0+\lambda_1u_1+\lambda_2u_2=\orig.
\end{equation}
Since $u_3=(0,\beta_3)\in\lineseg{u_1}{u_2}$ there exists some $0<\mu<1$ such that $\mu\alpha_1+(1-\mu)\alpha_2=0$. But $\lambda_1\alpha_1+\lambda_2\alpha_2=0$, hence
$$\frac{\lambda_1}{\lambda_1+\lambda_2}\alpha_1+\frac{\lambda_2}{\lambda_1+\lambda_2}\alpha_2=0.$$
By uniqueness of $\mu$,
\begin{equation}\label{eq:u3}
u_3=\frac{\lambda_1}{\lambda_1+\lambda_2}u_1+\frac{\lambda_2}{\lambda_1+\lambda_2}u_2.
\end{equation}

Similarly, since $u_0=(0,\beta_0)\in\lineseg{u_2}{u_4}$ there exists some $0<\nu<1$ such that $u_0=\nu u_2+(1-\nu)u_4$, giving
$$u_4=\frac{1}{1-\nu}u_0-\frac{\nu}{1-\nu}u_2.$$
Comparing coefficients we see that
\begin{equation}\label{eq:alpha1_1}
\alpha_1=-\frac{\nu}{1-\nu}\alpha_2.
\end{equation}
But $u_4=u_1+\kappa u_0$ for some $\kappa>0$. Combining this with equation~\eqref{eq:barycentre} we see that
$$u_4=\frac{\lambda_1\kappa-\lambda_0}{\lambda_1}u_0-\frac{\lambda_2}{\lambda_1}u_2.$$
Comparing coefficients, we obtain
\begin{equation}\label{eq:alpha1_2}
\alpha_1=-\frac{\lambda_2}{\lambda_1}\alpha_2.
\end{equation}
Equating equations~\eqref{eq:alpha1_1} and~\eqref{eq:alpha1_2} gives
\begin{equation}\label{eq:u4}
u_4=\frac{\lambda_1+\lambda_2}{\lambda_1}u_0-\frac{\lambda_2}{\lambda_1}u_2.
\end{equation}

Notice that, since both $u_0$ and $u_3$ are contained in $\linspan{w}$, there exists some $\gamma>0$ such that $-\gamma u_3=u_0$. Substituting into equation~\eqref{eq:u4} we have
\begin{equation}\label{eq:barycentre_intermediate}
\frac{\lambda_2}{\lambda_1}u_2+u_4+\gamma'u_3=\orig
\end{equation}
where $\gamma'=\gamma(\lambda_1+\lambda_2)/\lambda_1>0$. Substituting in equation~\eqref{eq:u3} we obtain
$$\frac{\lambda_2}{\lambda_1}u_2+u_4+\frac{\gamma'\lambda_1}{\lambda_1+\lambda_2}u_1+\frac{\gamma'\lambda_2}{\lambda_1+\lambda_2}u_2=\orig.$$
Using equation~\eqref{eq:u4} to rewrite the first two terms and clearing denominators gives:
\begin{equation}\label{eq:barycentre_3}
(\lambda_1+\lambda_2)^2u_0+\gamma'\lambda_1^2u_1+\gamma'\lambda_1\lambda_2u_2=\orig.
\end{equation}

Set $h:=\lambda_0+\lambda_1+\lambda_2$ and $\Gamma:=(\lambda_1+\lambda_2)^2+\gamma'\lambda_1^2+\gamma'\lambda_1\lambda_2$. By comparing equations~\eqref{eq:barycentre} and~\eqref{eq:barycentre_3}, uniqueness of barycentric coordinates gives:
\begin{align*}
h(\lambda_1+\lambda_2)^2&=\Gamma\lambda_0,\\
h\gamma'\lambda_1^2&=\Gamma\lambda_1,\\
h\gamma'\lambda_1\lambda_2&=\Gamma\lambda_2.
\end{align*}
In particular,
$$\gamma'=\frac{(\lambda_1+\lambda_2)^2}{\lambda_0\lambda_1}.$$
Substituting this expression for $\gamma'$ back into equation~\eqref{eq:barycentre_intermediate} gives
\begin{equation}\label{eq:barycentre_mutation}
\lambda_0\lambda_2u_2+(\lambda_1+\lambda_2)^2u_3+\lambda_0\lambda_1u_4=\orig.
\end{equation}

Finally, we consider the situation where $T_1\subset\NQ$ is the triangle associated with a fake weighted projective plane with weights $(\lambda_0,\lambda_1,\lambda_2)$, and assume that there exists a one-step mutation to some triangle $T_2\subset\NQ$. If $\lambda_0$ does not divide $(\lambda_1+\lambda_2)^2$, then by equation~\eqref{eq:barycentre_mutation} the associated weights are
$$\left(\lambda_0\lambda_1,\lambda_0\lambda_2,(\lambda_1+\lambda_2)^2\right),$$
and these fail to be well-formed when $\lambda_0>1$. Therefore, we must have $\lambda_0\divides(\lambda_1+\lambda_2)^2$, giving weights
$$\left(\lambda_1,\lambda_2,\frac{(\lambda_1+\lambda_2)^2}{\lambda_0}\right).$$
\end{proof}

\begin{remark}
Let $(\lambda_0,\lambda_1,\lambda_2)$ be well-formed weights such that $\lambda_0\divides(\lambda_1+\lambda_2)^2$, and suppose that there exists some prime $p$ such that
$$p\divides \lambda_1\quad\text{ and }\quad p\divides\frac{(\lambda_1+\lambda_2)^2}{\lambda_0}.$$
Then $p\divides\lambda_2^2$ and so $p\divides\lambda_2$. But this contradicts $(\lambda_0,\lambda_1,\lambda_2)$ being well-formed. Hence
$$\left(\lambda_1,\lambda_2,\frac{(\lambda_1+\lambda_2)^2}{\lambda_0}\right)$$
are also well-formed.
\end{remark}

\begin{example}\label{ex:P3511_rigid}
There exists no one-step mutation from $\Proj(3,5,11)$ to any other weighted projective space, since $3\notdivides(5+11)^2$, $5\notdivides(3+11)^2$, and $11\notdivides(3+5)^2$.
\end{example}

\begin{example}\label{ex:not_sufficient}
The requirement that $\lambda_0\divides(\lambda_1+\lambda_2)^2$ in Proposition~\ref{prop:weights_transform} is necessary but not sufficient. For example, consider the triangle $T=\sconv{(10,-7),(-5,2),(0,1)}\subset\NQ$. This has weights $(1,2,3)$, however there exist no one-step mutations from $T$.
\end{example}

\subsection{One-step mutations in $\NQ$ and $T$-singularities}\label{subsec:T_sings}
Our aim in this section is to characterise when a mutation exists. In order to do this, we require the definition of a $T$-singularity.

\begin{definition}[\protect{\cite[Definition~3.7]{KS88}}]
A quotient surface singularity is called a \emph{$T$-singularity} if it admits a $\Q$-Gorenstein one-parameter smoothing.
\end{definition}
$T$-singularities include the du Val singularities $\frac{1}{r}(1,r-1)$, and are cyclic quotient singularities of the form $\frac{1}{nd^2}(1,dna-1)$, where $\gcd{d,a}=1$~\cite[Proposition~3.10]{KS88}.

\begin{lemma}\label{lem:T_singularities}
An isolated quotient singularity $\frac{1}{r}(a,b)$ is a $T$-singularity if and only if $r\divides(a+b)^2$.
\end{lemma}
\begin{proof}
We begin by noting that the condition that $r\divides(a+b)^2$ is independent of the choice of representation of $\frac{1}{r}(a,b)$. For let $c$ be any integer coprime to $r$. Then $r\divides(a+b)^2$ if and only if $r\divides c^2(a+b)^2=(ca+cb)^2$.

Suppose we are given a $T$-singularity. Writing the singularity in the form $\frac{1}{nd^2}(1,dna-1)$ where $\gcd{d,a}=1$, we see that $nd^2\divides d^2n^2a^2$. Conversely consider the isolated quotient singularity $\frac{1}{r}(a,b)$. Since $a$ is invertible $\bmod\ r$, we can write this as $\frac{1}{r}(1,b'-1)$, where $b'\equiv ba^{-1}+1\ \modb{r}$. Write $r=nd^2$ where $n$ is square-free. Since $nd^2\divides b'^2$ by assumption, we see that $nd\divides b'$. In particular, we can express our singularity in the form $\frac{1}{nd^2}(1,dn\alpha-1)$ for some $\alpha\in\Z_{>0}$. Finally, we note that this really is a $T$-singularity: if $\gcd{d,\alpha}=c$ then we can absorb this factor into $n'=nc^2$ whilst rescaling $d'=d/c$ and $\alpha'=\alpha/c$.
\end{proof}

\begin{prop}\label{prop:quotient_sings}
Let $X$ be a fake weighted projective plane corresponding to a triangle $T\subset\NQ$, and suppose that the cone $C$ spanned by an edge $E$ of $T$ corresponds to a $\frac{1}{r}(a,b)$ singularity. There exists a one-step mutation to a fake weighted projective plane $Y$ given by $\mut_w(T,F)$ with $w(E)=\hmin$ if and only if $\frac{1}{r}(a,b)$ is a $T$-singularity.
\end{prop}
\begin{proof}
Let $X$ correspond to the lattice triangle $T=\sconv{v_1,v_2,v_3}\subset\NQ$, where $\orig\in\intr{T}$ and the vertices $\V{T}\subset N$ are all primitive. Consider the cone $C=\scone{v_1,v_2}$ spanned by the edge $E=\lineseg{v_1}{v_2}$; this is an isolated quotient singularity (possibly smooth), so is of the form $\frac{1}{r}(a,b)$ for some $r,a,b\in\Z_{>0}$, $\gcd{r,a}=\gcd{r,b}=1$.

Let $w\in M$ be a primitive lattice point such that $w(v_1)=w(v_2)=h$ for some $h<0$. Then, up to translation, there exists a factor $F\subset\NQ$, $w(F)=0$, such that $T':=\mut_w(T,F)$ is a triangle if and only if $v_1+(-h)F=E$. Equivalently, if and only if $h\divides\abs{E\cap N}-1$.

Finally, we express the values of $h$ and $\abs{E\cap N}-1$ in terms of the singularity $\frac{1}{r}(a,b)$. Set $k:=\gcd{r,a+b}$. Then the height $h=-r/k$, and the number of points on the edge $E$ is given by
$$\abs{\{m\mid m\in\{0,\ldots,r\}\text{ and }(a+b)m\equiv 0\ \modb{r}\}}=1+\frac{r}{h}=1+k.$$
Hence $h\divides\abs{E\cap N}-1$ if and only if $r/k\divides k$. But $r/k\divides k$ if and only if $r\divides\gcd{r,a+b}^2=\gcd{r^2,(a+b)^2}$, and $r\divides\gcd{r^2,(a+b)^2}$ if and only if $r\divides(a+b)^2$. The result follows by Lemma~\ref{lem:T_singularities}.
\end{proof}

\begin{example}\label{ex:not_T_singularities}
Returning to Example~\ref{ex:not_sufficient}, we see that the corresponding fake weighted projective space $X$ is a quotient of $\Proj(1,2,3)$ with $\mult{X}=5$. The three singularities are $\frac{1}{5}(1,3)$, $\frac{1}{10}(1,3)$, and $\frac{1}{15}(1,11)$, none of which is a $T$-singularity.
\end{example}

When $X$ is a weighted projective plane, Proposition~\ref{prop:quotient_sings} tells us that the condition that $\lambda_0\divides(\lambda_1+\lambda_2)^2$ in Proposition~\ref{prop:weights_transform} is both necessary and sufficient.

\subsection{One-step mutations and Diophantine equations}
Given the results of~\S\ref{subsec:weights} and~\S\ref{subsec:T_sings}, we are now in a position to relate one-step mutations of Fano triangles to solutions of certain Diophantine equations.

\begin{lemma}\label{lem:diophantine_derivation}
Let $(\lambda_0,\lambda_1,\lambda_2)\in\Z_{>0}^3$ with $d=\gcd{\lambda_0,\lambda_1,\lambda_2}$. Write:
\begin{enumerate}
\item\label{condn:lambda}
$\lambda_i=dc_ia_i^2$, where $a_i,c_i\in\Z_{>0}$ and $c_i$ is square-free;
\item\label{condn:degree}
$(\lambda_0+\lambda_1+\lambda_2)^2/(\lambda_0\lambda_1\lambda_2)=m^2/(rk^2)$, where $m,k,r\in\Z_{>0}$ and $r$ is square-free;
\item\label{condn:square_free}
$c_0c_1c_2 = gS^2$ and $dr = hT^2$, where $g,h,S,T\in\Z_{>0}$ and both $g$ and $h$ are square-free.
\end{enumerate}
Then $(da_0,da_1,da_2)$ is a solution to the Diophantine equation
\begin{equation}\label{eq:diophantine_general}
Smx_0x_1x_2 = Tk(c_0x_0^2 + c_1x_1^2 + c_2x_2^2).
\end{equation}
\end{lemma}
\begin{proof}
By substituting expressions~\eqref{condn:lambda} and~\eqref{condn:square_free} into~\eqref{condn:degree} we obtain
$$
gS^2m^2(da_0)^2(da_1)^2(da_2)^2 = hT^2k^2\left(c_0(da_0)^2 + c_1(da_1)^2 + c_2(da_2)^2\right)^2.
$$
Comparing square-free parts, we conclude that $g = h$. Cancelling and taking square-roots on both sides establishes the result.
\end{proof}

Since the weights are assumed to be well-formed, $d=S=T=1$ and equation~\eqref{eq:diophantine_general} becomes
\begin{equation}\label{eq:diophantine}
mx_0x_1x_2 = k(c_0x_0^2 + c_1x_1^2 + c_2x_2^2).
\end{equation}
Suppose that $(a_0,a_1,a_2)$ is a positive integral solution to equation~\eqref{eq:diophantine}, so that $\lambda_i=c_ia_i^2$. The expression
\begin{equation}\label{eq:wps_degree}
\frac{(\lambda_0+\lambda_1+\lambda_2)^2}{\lambda_0\lambda_1\lambda_2}
\end{equation}
occurring in Lemma~\ref{lem:diophantine_derivation} is equal to the degree of $\Proj(\lambda_0,\lambda_1,\lambda_2)$. More generally if $X$ is a fake weighted projective plane with weights $(\lambda_0,\lambda_1,\lambda_2)$ then~\eqref{eq:wps_degree} is equal to $\mult{X}(-K_X)^2$.

\begin{prop}\label{prop:mult_preserved}
Let $X$ be a fake weighted projective plane and suppose that there exists a one-step mutation to a fake weighted projective plane $Y$. Then the weights of $X$ and $Y$ give solutions to the same Diophantine equation~\eqref{eq:diophantine}. In particular, $\mult{X}=\mult{Y}$.
\end{prop}
\begin{proof}
With notation as in Lemma~\ref{lem:diophantine_derivation}, we can write the weights $(\lambda_0,\lambda_1,\lambda_2)$ of $X$ in the form $\lambda_i=c_ia_i^2$, where the $c_i$ are square-free positive integers. From Proposition~\ref{prop:weights_transform} we know that $Y$ has weights
$$\left(\lambda_1,\lambda_2,\frac{(\lambda_1+\lambda_2)^2}{\lambda_0}\right)=\left(c_1a_1^2,c_2a_2^2,\frac{(c_1a_1^2+c_2a_2^2)^2}{c_0a_0^2}\right).$$
The final weight is an integer; in particular, it has square-free part $c_0$. Thus the $c_i$ are invariant under mutation. Furthermore,
\begin{align*}
\frac{\left(\lambda_1+\lambda_2+\frac{(\lambda_1+\lambda_2)^2}{\lambda_0}\right)^2}{\lambda_1\cdot\lambda_2\cdot\frac{(\lambda_1+\lambda_2)^2}{\lambda_0}}&=\frac{\left(\lambda_0\lambda_1+\lambda_0\lambda_2+(\lambda_1+\lambda_2)^2\right)^2}{\lambda_0\lambda_1\lambda_2(\lambda_1+\lambda_2)^2}\\
&=\frac{(\lambda_0+\lambda_1+\lambda_2)^2}{\lambda_0\lambda_1\lambda_2}\\
&=\frac{m^2}{rk^2}
\end{align*}
and so the ratio $m/k$ is also preserved by mutation. Hence the weights of $X$ and of $Y$ both generate solutions to the same Diophantine equation~\eqref{eq:diophantine}.

Finally we recall that degree is fixed under mutation, hence $(-K_X)^2=(-K_Y)^2$. But
$$\frac{m^2}{rk^2}=\mult{X}(-K_X)^2=\mult{Y}(-K_Y)^2$$
and so $\mult{X}=\mult{Y}$.
\end{proof}

By combining Propositions~\ref{prop:weights_transform},~\ref{prop:quotient_sings}, and~\ref{prop:mult_preserved} we obtain Proposition~\ref{prop:only_mutate_T_sings}.

\begin{remark}
The weights of a fake weighted projective plane correspond to a solution $(a_0,a_1,a_2)$ of equation~\eqref{eq:diophantine}. A one-step mutation gives a second solution via the transformation:
$$(a_0,a_1,a_2)\mapsto\left(\frac{m}{k}\frac{a_1a_2}{c_0}-a_0,a_1,a_2\right).$$
\end{remark}

\begin{example}\label{ex:mutations_of_P2}
Consider $\Proj^2$. In this case $m/k=3$, $c_0=c_1=c_2=1$, and $(1,1,1)\in\Z_{>0}^3$ is a solution of
\begin{equation}\label{eq:diophantine_P2}
3x_0x_1x_2=x_0^2+x_1^2+x_2^2.
\end{equation}
Up to isomorphism, there is a single one-step mutation to $\Proj(1,1,4)$, giving a solution $(1,1,2)\in\Z_{>0}^3$ of equation~\eqref{eq:diophantine_P2}. Proceeding in this fashion we obtain a graph of one-step mutations corresponding to solutions of~\eqref{eq:diophantine_P2}, which we illustrate to a depth of five mutations:
\begin{center}
\includegraphics[scale=0.85]{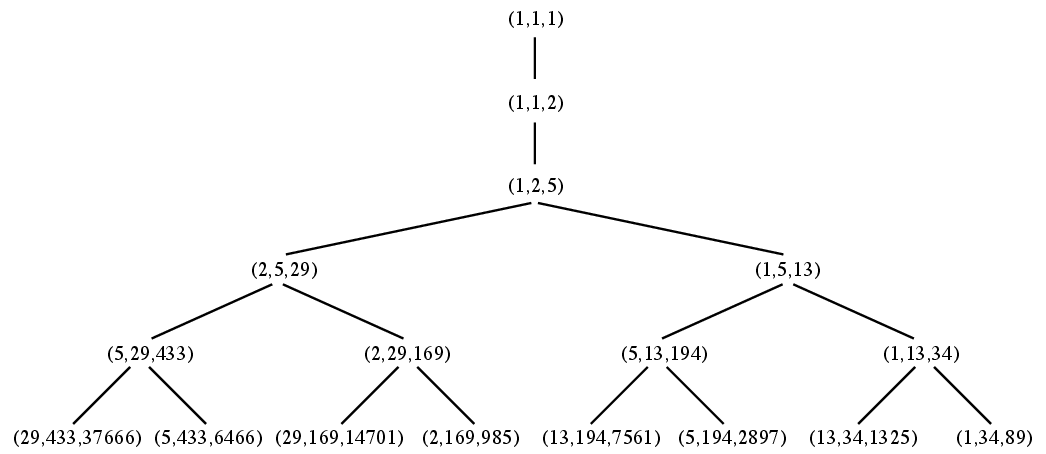}
\end{center}
\end{example}

\begin{definition}\label{defn:minimal}
The \emph{height} of the weights $(\lambda_0,\lambda_1,\lambda_2)$ is given by the sum $h:=\lambda_0+\lambda_1+\lambda_2\in\Z_{>0}$. We call the weights \emph{minimal} if for any sequence of one-step mutations $(\lambda_0,\lambda_1,\lambda_2)\mapsto\ldots\mapsto(\lambda'_0,\lambda'_1,\lambda'_2)$ we have that $h\leq h'$.
\end{definition}

\begin{lemma}\label{lem:reduce_height}
Given weights $(\lambda_0,\lambda_1,\lambda_2)$ at height $h$ there exists at most one one-step mutation such that $h'\le h$. Moreover, if $h'=h$ then the weights are the same.
\end{lemma}
\begin{proof}
Without loss of generality suppose we have two one-step mutations
$$\left(\lambda_1,\lambda_2,\frac{(\lambda_1+\lambda_2)^2}{\lambda_0}\right)\qquad\text{ and }\qquad\left(\lambda_0,\frac{(\lambda_0+\lambda_2)^2}{\lambda_1},\lambda_2\right)$$
with respective heights $h'$ and $h''$ such that $h'\le h$ and $h''\le h$. Since $h'\le h$ we obtain $(\lambda_1+\lambda_2)^2\leq\lambda_0^2$, and so
\begin{equation}\label{eq:h'_lt_h}
\lambda_1^2+\lambda_2^2<\lambda_0^2.
\end{equation}
From $h''\le h$ we obtain
\begin{equation}\label{eq:h''_lt_h}
\lambda_0^2+\lambda_2^2<\lambda_1^2.
\end{equation}
Combining equations~\eqref{eq:h'_lt_h} and~\eqref{eq:h''_lt_h} gives a contradiction, hence there exists at most one one-step mutation such that $h'\leq h$. If we suppose that $h'=h$ then
$$\frac{(\lambda_1+\lambda_2)^2}{\lambda_0}=\lambda_0$$
and equality of the weights is immediate.
\end{proof}

The height imposes a natural direction on the graph of all one-step mutations generated by the weight $(\lambda_0,\lambda_1,\lambda_2)$. Lemma~\ref{lem:reduce_height} tells us that this directed graph is a tree, with a uniquely defined minimal weight.

\section{Example: An infinite number of minimal weights}\label{sec:infinite_min_weights}
In this section we shall focus on the Diophantine equation
\begin{equation}\label{eq:3_5_7}
12x_0x_1x_2=3x_0^2+5x_1^2+7x_2^2.
\end{equation}
Any solution $(a_0,a_1,a_2)$ such that $(3a_0^2,5a_1^2,7a_2^2)$ is well-formed corresponds to weighted projective space $\Proj(3a_0^2,5a_1^2,7a_2^2)$ of degree $144/105$. One possible such solution is $(2,1,1)$ giving $\Proj(12,5,7)$. Consider the graph $\mathcal{G}$ of all such solutions. Two solutions lie in the same component if and only if there exists a sequence of one-step mutations between the corresponding weighted projective planes. Furthermore, each component is a tree with unique minimal weight. We shall show that there exists an infinite number of components, and that every component contains at most two solutions; in fact the only component with a single solution is $(2,1,1)$.

\subsection{Coprime solutions give well-formed weights}
Let $(a_0,a_1,a_2)$ be a solution of equation~\eqref{eq:3_5_7} such that $\gcd{a_0,a_1,a_2}=1$. Clearly this is a necessary condition for the corresponding weights $(3a_0^2,5a_1^2,7a_2^2)$ to be well-formed. We shall show that it is sufficient. For suppose that there exists some prime $p$ such that $p\divides c_ia_i^2$ and $p\divides c_ja_j^2$, $i\ne j$. Since $p$ cannot simultaneously divide both $c_i$ and $c_j$, we have that $p$ must divide either $a_i$ or $a_j$. In particular, $p\divides 12a_0a_1a_2$ and so, by equation~\eqref{eq:3_5_7}, $p$ divides the remaining weight $c_ka_k^2$. Similarly, since $p$ can divide at most one of $3$, $5$, and $7$ we see that $p^2\divides 12a_0a_1a_2$ and so $p^2$ divides each of the three weights. We conclude that $p\divides\gcd{a_0,a_1,a_2}$, contradicting coprimality.

\subsection{A necessary and sufficient condition for rational solutions when $a_1$ and $a_2$ are fixed}
Fix $a_1,a_2\in\Z_{>0}$ and consider the quadratic
\begin{equation}\label{eq:3_5_7_quadratic}
12xa_1a_2=3x^2+5a_1^2+7a_2^2.
\end{equation}
The discriminant is given by
$$12^2a_1^2a_2^2-12(5a_1^2+7a_2^2)=12\left(5a_1^2(a_2^2-1)+7a_2^2(a_1^2-1)\right),$$
which is always non-negative. The discriminant is zero only in the case $a_1=a_2=1$, corresponding to the solution $(2,1,1)$ of equation~\eqref{eq:3_5_7}. Furthermore, we see that a rational solution to equation~\eqref{eq:3_5_7_quadratic} exists if and only if
\begin{equation}\label{eq:3_5_7_condition}
5a_1^2(a_2^2-1)+7a_2^2(a_1^2-1)=3N^2,\qquad\text{ for some }N\in\Z_{>0}.
\end{equation}

\subsection{Any rational solution is an integral solution}
Suppose that $\alpha,\beta\in\R$ are the two solutions of equation~\eqref{eq:3_5_7_quadratic}. We obtain:
\begin{align}
\label{eq:3_5_7_sols_are_integer_1}\alpha+\beta&=4a_1a_2,\\
\label{eq:3_5_7_sols_are_integer_2}3\alpha\beta&=5a_1^2+7a_2^2.
\end{align}
In particular, since the right-hand side in each case is a strictly positive integer, we see that $\alpha,\beta>0$. Furthermore, $\alpha$ is rational if and only if $\beta$ is rational. Since we are only interested in rational solutions, we can assume that both $\alpha$ and $\beta$ are rational. Let us write
$$\alpha=\frac{n_1}{m_1}\qquad\text{ and }\qquad\beta=\frac{n_2}{m_2},$$
where the fractions are expressed in their reduced form, i.e.~$\gcd{n_i,m_i}=1$. Then
\begin{align}
\label{eq:3_5_7_rational_is_integer_1}
m_1m_2&\divides 3n_1n_2,\\
\label{eq:3_5_7_rational_is_integer_2}
m_1m_2&\divides n_1m_2+n_2m_1.
\end{align}
By~\eqref{eq:3_5_7_rational_is_integer_2}, $m_2\divides m_1$ and $m_1\divides m_2$, forcing $m_1=m_2$. Without loss of generality, from~\eqref{eq:3_5_7_rational_is_integer_1} we may assume that $m_1\divides 3n_2$ and $m_2\divides n_1$. But then $m_1\divides n_1$, forcing $m_1=m_2=1$. Hence $\alpha,\beta\in\Z_{>0}$.

\subsection{The values $a_1$ and $a_2$ are fixed under one-step mutations}
We now show that, given a solution $(a_0,a_1,a_2)$ such that $\gcd{a_0,a_1,a_2}=1$, the values of $a_1$ and $a_2$ are fixed under one-step mutation. For suppose that
\begin{equation}\label{eq:3_5_7_mutate_about_5}
\frac{(3a_0^2+7a_2^2)^2}{5a_1^2}\in\Z.
\end{equation}
Without loss of generality we may take $\alpha=a_0$. We see that $5\divides 3a_0^2+7a_2^2=3\alpha^2+3\alpha\beta-5a_1^2$ by~\eqref{eq:3_5_7_sols_are_integer_2}, hence $5\divides3\alpha(\alpha+\beta)=12a_0a_1a_2$ by~\eqref{eq:3_5_7_sols_are_integer_1}. Since the weights are pairwise coprime, the only possibility is that $5\divides a_1$. Returning to equation~\eqref{eq:3_5_7_mutate_about_5} we see that $5^2\divides 3a_0^2+7a_2^2$, and proceeding as before we find that $5^2\divides a_1$. Clearly we can repeat this process an arbitrary number of times, increasing the power of $5$ at each step. This is a contradiction. The case when
$$\frac{(3a_0^2+5a_1^2)^2}{7a_2^2}\in\Z$$
is dealt with similarly.

\subsection{An infinite number of components}
Set $a_1=1$ in condition~\eqref{eq:3_5_7_condition}. The condition becomes $a_2^2-1=15M^2$, where $5M=N$. This is a Pell equation, and Emerson~\cite{Eme69} has shown that there exists an infinite number of integer solutions given by a recurrence relation. In this case we see that $a_2^{(n)}$ and $M^{(n)}$ are generated by:
\begin{align*}
a_2^{(0)}&=1,&M^{(0)}&=0,\\
a_2^{(1)}&=4,&M^{(1)}&=1,\\
a_2^{(n+1)}&=8a_2^{(n)}-a_2^{(n-1)},&M^{(n+1)}&=8M^{(n)}-M^{(n-1)}.
\end{align*}
Substituting these expressions back into the original quadratic~\eqref{eq:3_5_7_quadratic} gives:
$$a_0^{(n+1)}=2a_2^{(n)}\pm 5M^{(n)}.$$
These solutions are coprime (since $a_1=1$) and so correspond to well-formed weights. We will focus on the smaller of the two solutions, corresponding to the minimum of the two weights. Substituting the expressions for $a_2^{(n)}$ and $M^{(n)}$ gives:
\begin{align*}
a_0^{(n+1)}&=2a_2^{(n+1)}-5M^{(n+1)}\\
&=8\left(2a_2^{(n)}-5M^{(n)}\right)-\left(2a_2^{(n-1)}-5M^{(n-1)}\right)\\
&=8a_0^{(n)}-a_0^{(n-1)}.
\end{align*}
Hence we obtain the recurrence relation:
\begin{align*}
a_0^{(0)}&=2,\\
a_0^{(1)}&=3,\\
a_0^{(n+1)}&=8a_0^{(n)}-a_0^{(n-1)}.
\end{align*}

\begin{remark}
If instead we insist that $a_2=1$, we obtain the Pell equation $a_1^2-1=21M^2$, where $7M=N$. In this case the recurrence relation is given by:
\begin{align*}
a_1^{(0)}&=1,&M^{(0)}&=0,\\
a_1^{(1)}&=55,&M^{(1)}&=12,\\
a_1^{(n+1)}&=110a_1^{(n)}-a_1^{(n-1)},&M^{(n+1)}&=110M^{(n)}-M^{(n-1)}.
\end{align*}
Proceeding as above we find that
\begin{align*}
a_0^{(0)}&=2,\\
a_0^{(1)}&=26,\\
a_0^{(n+1)}&=110a_0^{(n)}-a_0^{(n-1)}.
\end{align*}
Hence we have a second infinite family of components of $\mathcal{G}$. Notice that these two families do not exhaust all the possibilities: for example, $a_1=5$, $a_2=4$ satisfies condition~\eqref{eq:3_5_7_condition}, giving the two solutions $(1,5,4)$ and $(79,5,4)$.
\end{remark}

\subsection*{Acknowledgments}
Our thanks to Tom Coates, Alessio Corti, and Song Sun for many useful conversations. The authors are supported by EPSRC grant EP/I008128/1.

\providecommand{\bysame}{\leavevmode\hbox to3em{\hrulefill}\thinspace}
\providecommand{\MR}{\relax\ifhmode\unskip\space\fi MR }
\providecommand{\MRhref}[2]{%
  \href{http://www.ams.org/mathscinet-getitem?mr=#1}{#2}
}
\providecommand{\href}[2]{#2}

\end{document}